\RequirePackage{fix-cm}
\documentclass[smallextended]{svjour3}       
\smartqed  
\usepackage{graphicx}
\usepackage{amssymb}

\usepackage{amsmath}

\begin{document}

\title{A note on the Ramsey number of even wheels versus stars
}


\author{Sh. Haghi         \and
        H. R. Maimani 
}


\institute{Sh. Haghi \at
             Mathematics Section, Department of Basic Sciences, Shahid Rajaee Teacher Training University, P. O. BOX 16783-163,Tehran, Iran \\
             School of Mathematics, Institute for Research in Fundamental Sciences (IPM), P. O. BOX 19395-5746, Tehran, Iran\\
              Tel.: +982122970060\\
              \email{sh.haghi@example.com}           
           \and
           H.R. Maimani \at
              Mathematics Section, Department of Basic Sciences, Shahid Rajaee Teacher Training University, P. O. BOX 16783-163,Tehran, Iran\\
              School of Mathematics, Institute for Research in Fundamental Sciences (IPM), P. O. BOX 19395-5746, Tehran, Iran
}

\date{Received: date / Accepted: date}

\maketitle

\begin{abstract}
For two graphs $G_1$ and $G_2$, the \textit{Ramsey number} $R(G_1,G_2)$ is the smallest integer $N$, such that for any graph on $N$ vertices, either $G$ contains $G_1$ or $\overline{G}$ contains $G_2$. Let $S_n$ be a \textit{star} of order $n$ and $W_m$ be a \textit{wheel} of order $m+1$.
In this paper, it is shown that $R(W_n, S_n)\leq{5n/2-1}$, where $n\geq{6}$ is even. It was proven a theorem which implies that $R(W_n,S_n)\geq{5n/2-2}$, where $n\geq{6}$ is even. Therefore we conclude that $R(W_n,S_n)=5n/2-2$ or $5n/2-1$, for $n\geq{6}$ and even.
\end{abstract}
\section{Introduction and Background}
Let $G=(V, E)$ denote a finite simple graph on the vertex set $V$ and the edge set $E$. The subgraph of $G$ \textit{induced} by $S\subseteq{V}$, $G[S]$, is a graph with vertex set $S$ and two vertices of $S$ are adjacent in $G[S]$ if and only if they are adjacent in $G$. The complement of a graph $G$, which is denoted by $\overline{G}$, is the graph with vertex set $V(G)$ and two vertices in $\overline{G}$ are adjacent if and only if they are not adjacent in $G$. For a vertex $v \in V(G)$, we denote the set of all neighbors of $v$ by $N_G(v)$ ( or $N(v)$). The degree of a vertex $v$ in a graph $G$, denoted by $deg_G(v)$ ( or $deg(v)$), is the size of the set $N(v)$.
a vertex of a connected graph is a \textit{cut-vertex} if its removal produces 
a disconnected graph.

The graph $K_n$ is the complete graph on $n$ vertices, and $C_n$ is the cycle graph on $n$ vertices.
The minimum degree,
maximum degree and clique number of $G$ are denoted by $\delta(G)$, $\Delta(G)$ and $\omega(G)$, respectively. The \textit{girth} of graph $G$, $g(G)$, is the length of shortest cycle. 
Also, the \textit{circumference} of graph $G$, is the length of longest cycle in $G$ and denoted by $c(G)$.
A graph $G$ of order $n$ is called \textit{Hamiltonian}, \textit{pancyclic} and \textit{weakly pancyclic} if it contains $C_n$, cycles of every length between
$3$ and $n$, and cycles of every length $l$ with $g(G) \leq{l}\leq{c(G)}$, respectively.
We say that $G$ is a
\textit{join} graph if $G$ is the complete union of two graphs
$G_1=(V_1,E_1)$ and $G_2=(V_2,E_2)$. In other words, $V=V_1\cup V_2$
and $E=E_1\cup E_2\cup\{uv : u\in{V_1}, v\in{V_2}\}$. If $G$ is the
join graph of $G_1$ and $G_2$, we shall write $G=G_1+G_2$. Let $G_1$ and $G_2$ be two graphs with vertex sets $V_1$ and $V_2$ and edge sets $E_1$ and $E_2$, respectively.
 A $\textit{wheel}$ $W_m$ is a graph on $m+1$ vertices obtained from $C_m$ by adding one vertex which is called the $\textit{hub}$ and joining each vertex of $C_m$ to the hub with the edges called the $\textit{rim}$ of the wheel. In other words, $W_m=C_m+K_1$.  A \textit{star} $S_n$ is the complete bipartite graph $K_{1,n-1}$.\\
A (proper) coloring is a function
$c : V (G)\longrightarrow{\mathbb{N}}$ (where $\mathbb{N}$ is the set of positive integers) such that $c(u)\neq{c(v)}$ if $u$ and
$v$ are adjacent in $G$. A graph $G$ is \textit{$k$-colorable} if there exists a coloring of $G$ from a set of $k$ colors. The minimum
positive integer $k$ for which $G$ is $k$-colorable is the \textit{chromatic number} of $G$ and is
denoted by $\chi(G)$.\\
For two graphs $G_1$ and $G_2$, the Ramsey number $R(G_1,G_2)$ is the smallest positive integer
$N$ such that for every graph $G$ on $N$ vertices, $G$ contains $G_1$ as a subgraph or the complement
of $G$ contains $G_2$ as a subgraph.\\
Harary et.al in \cite{H} proved the following lower bound for Ramsey numbers:\\
$R(G,H)\geq{(\chi(G)-1).(l(H)-1)+1}$,\\
where $l(H)$ is the number of vertices in the largest
connected component of $H$.\\
In this note we consider the Ramsey number for stars versus wheels. There are a lot of results about this subject.
The Harary lower bound for $R(W_m, S_n)$ is $3n-2$ or $2n-1$, where $m$ is odd or even, respectively. 

There are many results about this Ramsey number when $m$ is odd.
Chen et al. in \cite{C}
proved that if $m\leq{n+1}$ and $m$ is odd, then $R(W_m,S_n) = 3n-2$. 
Hasmawati et al. in \cite{HB} also showed
that $R(W_m,S_n) = 3n-2$, for the case $m\leq{2n-3}$. But, one can see in \cite{Ha}, if $n\geq{2}$ and $m\geq{2n-2}$, then $R(W_m,S_n)=n+m-1$, where $m$ is odd.\\
Also, one can find many results about $R(W_m,S_n)$ when $m$ is even.

It was shown in \cite{S2} that $R(W_4, S_n)=2n-1$ if $n\geq{3}$ and odd, and $R(W_4, S_n)=2n+1$
if $n\geq{4}$ and even.
Korolova in \cite{K} proved that: \\
$R(W_m, S_n)\geq{2n + 1}$ for all
$n\geq{m}\geq{6}$ and $m$ even. Also,
Chen et al. in \cite{C} showed that $R(W_6, S_n)=2n+1$.
\\
It was proven in \cite{Y} that $R(W_8,S_n)=2n+2$ for $n\geq{6}$ and even. Also, it was Shown in \cite{Z} that $R(W_8,S_n)=2n+1$ for $n\geq{5}$ and odd.\\
Li et al. in \cite{BS} indicated two following theorems in which they obtained a new lower bound and showed that for some cases this bound is sharp. 
\begin{theorem}\label{tn}\cite{BS}
If $6\leq{m}\leq{2n-4}$ and m is even, then 
\begin{equation*}
 R(W_m,S_n)\geq \left\{\begin{array}{rl}
2n+m/2-3\qquad\qquad\quad & \mbox{if $n$ is odd and $m/2$ is even}\\
2n+m/2-2\qquad\qquad\quad & \mbox{otherwise}.
\end{array}\right.
\end{equation*}
\end{theorem}
\begin{theorem}\cite{BS}
If $n + 1\leq{m}\leq{2n-4}$ and $m$ is even, then
\begin{equation*}
 R(W_m,S_n)= \left\{\begin{array}{rl}
2n+m/2-3\qquad\qquad\quad & \mbox{if $n$ is odd and $m/2$ is even}\\
2n+m/2-2\qquad\qquad\quad & \mbox{otherwise}.
\end{array}\right.
\end{equation*}
\end{theorem}
But for some cases, $R(W_m,S_n)$ where $m$ is even, is still open. one of these cases is when $m=n$. 
It was shown in \cite{K} that $R(W_n, S_n)\leq{3n-3}$ when $n$ is even.
In this paper, we want to improve this upper bound and prove that:\\
\begin{theorem}\label{t1}
$R(W_n, S_n)\leq{5n/2-1}$, where $n\geq{6}$ is even.
\end{theorem}
\section{Preliminary Lemmas and Theorems}
To prove Theorem \ref{t1}, we need some theorems and lemmas.
\begin{lemma}\label{l1}
(Brandt et al. \cite{B}). Every non-bipartite graph $G$ of order $n$ with ${\delta(G)\geq{(n+2)/3}}$ is weakly pancyclic with $g(G) = 3$ or $4$.
\end{lemma}
\begin{lemma}\label{l3}
(Dirac \cite{D}). Let $G$ be a $2$-connected graph of order $n\geq{3}$ with $\delta(G) =\delta$. Then $c(G)\geq{min\{2\delta, n\}}$.
\end{lemma}
\begin{theorem}\label{t2}(Faudree and Schelp \cite{F}, Rosta \cite{R})
\begin{equation*}
 R(C_n,C_m)=\left\{\begin{array}{rl}
2n-1\qquad\qquad\quad & \mbox{for \,\, $3\leq{m}\leq{n}$,\,\, $m$ odd $(n,m)\neq{(3,3)} $}\\
n+m/2-1\qquad\quad & \mbox{for\,\, $4\leq{m}\leq{n}$,\,\, $m, n$  even $(n,m)\neq{(4,4)}$}\\
max\{n+m/2-1, 2m-1\}& \mbox{for \,\,  $4\leq{m}<{n}$,\,\, $m$ even and $n$ odd}.
\end{array}\right.
\end{equation*}
\end{theorem}
\begin{lemma}\label{l4}\cite{GO}
Let $G$ be a bipartite graph of order $n$ ($n$ even) with bipartition $(X,Y)$ and $|X|=|Y|=n/2$. If for all distinct nonadjacent vertices $u\in{X}$ and $v\in{Y}$, 
 we have $deg(u)+deg(v)>{n/2}$, then $G$ is Hamiltonian.
\end{lemma}

\section{Proof of the Theorem \ref{t1}}
From now on, let $G$ be a graph of order $N=5n/2-1$ where $n\geq{6}$ and is even, such that neither $G$ contains $W_n$ nor it's complement, $\overline{G}$, contains $S_n$. Also, for every vertex $t\in{V(G)}$ consider $H_t=G[N(t)]$. Since $\overline{G}$ has no $S_n$, $deg_{\overline{G}}(v)\leq{n-2}$. Thus, $\delta(G)\geq{3n/2}$. In the middle of the proof, we sometimes interrupt it and have some lemmas.\\
 Let $v_0\in{V(G)}$ be an arbitrary vertex. there exists a $k\in\{0, 1, 2, \dots, n-2\}$ such that $deg_G(v_0)=3n/2+k$ since  $\delta(G)\geq{3n/2}$. Thus, the order of $H_{v_0}=G[N(v_0)]$ is $3n/2+k$. By the second part of the Theorem \ref{t2}, we have $|V(H_{v_0})|=3n/2+k\geq{R(C_n,C_s)}$, where $s=2t$, and $t$ is an integer such that $4\leq{2t}\leq{n+k+1}$. (Note that in Theorem \ref{t1}, we have  $n\geq{6}$, so the case $(n,s)=(4,4)$ does not occur for $R(C_n,C_s)$ in Theorem \ref{t2}). Thus, either $H_{v_0}$ contains $C_n$ or $\overline{H}_{v_0}$ contains $C_s$. But if $H_{v_0}$ contains $C_n$, then $G$ contains $W_n$, which is a contradiction. Hence we have the following corollary.
\begin{corollary}\label{c1}
Let $v\in{V(G)}$ and $k$ be an element in the set $\{0,1,\ldots, n-2\}$ such that $|V(H_{v})|=3n/2+k$. Then $\overline{H}_v$ contains $C_{2t}$ for all integers $t$ such that  $4\leq{2t}\leq{n+k+1}$.

\end{corollary}
 \begin{proposition}\label{p1}
 $\omega(\overline{G})\leq{n-2}$ and $\omega(G)\leq{n-1}$.
 \end{proposition}
 \begin{proof}
 It is clear that $\omega(\overline{G})\leq{n-1}$ since $\Delta(\overline{G})\leq{n-2}$. Suppose $\omega(\overline{G})=n-1$ and $T=\{v_1,\ldots,v_{n-1}\}$ is a clique in $\overline{G}$. For any $v\in{V-T}$, $N_{\overline{G}}(v)\cap T=\emptyset$ otherwise $\overline{G}[T\cup\{v\}]$ contains $S_n$. Now consider $v\in{V-T}$ and $k$ be an element in the set $\{0,1,\ldots, n-2\}$ such that $|V(H_{v})|=3n/2+k$. The subgraph $\overline{H}_v$ is a disconnected graph with a connected component $\overline{G}[T]$. On the other hand, by Corollary \ref{c1}, $\overline{H}_v$ contains a cycle $C$ of length $2t$ where $t=\lfloor (n+k+1)/2 \rfloor$. Note that $C\nsubseteq{T}$, since $2t>n-1$. Thus, $C\subseteq{\overline{H}_v-T}$. But $\overline{H}_v-T$ has $n/2+k+1$ vertices, which is a contradiction. Hence $\omega(\overline{G})\leq{n-2}$. For the second part, assume to the contrary, $G$ contains $K_n$ and $H=G[V-K_n]$. Then $|N_G(v)\cap K_n|\geq{2}$ for all $v\in{V(H)}$, otherwise $deg_{\overline{G}}(v)\geq{n-1}$, which is a contradiction. If $|N_G(v)\cap K_n|=2$ for all $v\in{V(H)}$, then $H=K_{3n/2-1}$ since $\delta(G)\geq{3n/2}$. But $K_{3n/2-1}$ contains $W_n$, a contradiction. So, there is a vertex $u\in{V(H)}$ such that $|N_G(u)\cap K_n|\geq{3}$. But $\{u\}\cup K_n$ contains $W_n$, which is a contradiction. Thus,  $\omega(G)\leq{n-1}$.
 \end{proof}

We can divide the proof into some cases and subcases:\\

 \textbf{Case 1.} There is a vertex $v\in{V(G)}$ for which $H_v$ is bipartite.
 
 Let $H_v$ be a bipartite graph with bipartition $(X_v,Y_v)$ of order $3n/2+k$ such that $k\in\{0,1,\ldots,n-2\}$. Without loss of generality, suppose that $|X_v|\leq{|Y_v|}$.
   Thus, by Proposition \ref{p1}, we have $n/2+k+2\leq{|X_v|}\leq{3n/4+k/2}$ and $3n/4+k/2\leq{|Y_v|}\leq{n-2}$.\\
  Let $|X_v|=n/2+s$, where $s$ is an integer such that $k+2\leq{s}\leq{n/4+k/2}$, then $|Y_v|=n+k-s$. Since $\Delta(\overline{G})\leq{n-2}$ and $|H_v|=3n/2+k$, we conclude $\delta(H_v)\geq{n/2+k+1}$. Let $X_{v}^{\prime}$ and $Y_{v}^{\prime}$ obtained from $X_v$ and $Y_v$ by deleting $s$ and $n/2+k-s$ arbitrary vertices, respectively, and let $H_{v}^{\prime}=(X_{v}^{\prime},Y_{v}^\prime)$. Thus, $|X_{v}^{\prime}|=|Y_{v}^{\prime}|=n/2$  and $\delta(X_{v}^{\prime})\geq{s+1}$ and $\delta(Y_{v}^{\prime})\geq{n/2+k+1-s}$ in $H_{v}^{\prime}$. Hence for each two vertices $u_1\in{X_{v}^\prime}$ and $u_2\in{Y_{v}^\prime}$, 
 we have $deg(u_1)+deg(u_2)\geq{n/2+k+2}$ and by Lemma \ref{l4}, $H_{v}^{\prime}$ contains $C_n$. It means that $G$ contains $W_n$, which is a contradiction.\\
 
\textbf{Case 2.} For every vertex $t\in{V(G)}$, $H_t$ is non-bipartite.\\

\textbf{Subcase 2.1.} Suppose $H_t$ is disconnected for all $t\in{V(G)}$.

Let $t\in{V(G)}$ be an arbitrary vertex and $|V(H_t)|=3n/2+k$, where $k\in\{0, 1, 2, \dots, n-2\}$. We show that $H_t$ has exactly two connected components. Suppose to the contrary, $H_1$, $H_2$ and $H_3$ are three connected components of $H_t$. Since $\delta(H_t)\geq{n/2+k+1}$, we conclude $\delta(H_i)\geq{n/2+k+1}$ for $i=1, 2, 3$. Hence $|V(H_t)|>3n/2+k$, which is a contradiction. Now, let $X_t$, $Y_t$ be the set of vertices of two components of $H_t$. Assume that $|X_t|\leq{|Y_t|}$. We choose two adjacent vertices $u$ and $v$ in $Y_t$ since $\delta(H_t)\geq{n/2+k+1}$. Let $|V(H_u)|=3n/2+k^{\prime}$ and $|V(H_v)|=3n/2+k^{\prime\prime}$, where $k^{\prime}, k^{\prime\prime}\in\{0, 1, 2, \dots, n-2\}$. Also, let $X_u$, $Y_u$ and $X_v$, $Y_v$ be the set of vertices of two components of $H_u$ and $H_v$, respectively. Since, $H_t$ and $H_u$ are disconnected, $X_u$ or $Y_u$ is disjoint from $X_t$ and $Y_t$. To see this, with no loss of generality, suppose that $v$ is contained in $Y_u$. Thus, $t\in{Y_u}$ and hence $X_u\cap Y_t=X_u\cap X_t=\emptyset$. Similarly, $X_v$ or $Y_v$, say $X_v$, is disjoint from $X_t$ and $Y_t$. Thus, we have $Y_t\cap X_u=Y_t\cap X_v=X_t\cap X_u=X_t\cap X_v=\emptyset$. Also, $X_u\cap X_v=\emptyset$ otherwise if $l\in{X_u\cap X_v}$, then $l$ is adjacent to both $u$ and $v$. But $u\in{Y_v}$ implies that $l\in{Y_v}$. It means, $X_v\cap Y_v\neq\emptyset$ which is a contradiction. Thus, $X_u\cap X_v=\emptyset$.  Hence $|V(G)|\geq{|V(H_t)|+|X_u|+|X_v|}$ which means  $|V(G)|\geq{(3n/2+k)+(n/2+k^{\prime}+2)+(n/2+k^{\prime\prime}+2)}>5n/2-1$,
 which is a contradiction.\\

\textbf{Subcase 2.2.} Suppose $H_t$ is connected for some $t\in{V(G)}$.\\
 Assume that there exists a vertex $u\in{V(G)}$ for which $H_u$ is 2-connected and $|V(H_u)|=3n/2+k$ for some $k\in\{0,1,2,\ldots,n-2\}$. Thus, $\delta(H_u)\geq{n/2+k+1}\geq{(3n/2+k+2)/3}$ and by Lemma \ref{l1}, $H_u$ is weakly pancyclic with $g(G)=3$ or $4$. Also, by Lemma \ref{l3}, $c(H_u)\geq{min\{2\delta, 3n/2+k\}}$. Hence $c(H_u)\geq{n}$ which implies that $H_u$ contains $C_n$, a contradiction.\\
Now, assume each connected $H_t$ contains a cut-vertex.
Let $u$ be a cut-vertex of $H_t$ and $|V(H_t)|=3n/2+k$. We show that $H_t-u$ has exactly two connected components. Suppose to the contrary, $H_1$, $H_2$ and $H_3$ are three connected components of $H_t-u$. Since $\delta(H_t)\geq{n/2+k+1}$, $\delta(H_i)\geq{n/2+k}$ for $i=1, 2, 3$. Hence $|H_t|>3n/2+k$, which is a contradiction. Now, let $s_1$ be a cut-vertex of $H_t$ and $X_t$, $Y_t$ be the set of vertices of two components of $H_t-s_1$. Assume that $|X_t|\leq{|Y_t|}$. We choose two adjacent vertices $u$ and $v$ in $Y_t$ since $\delta(H_t)\geq{n/2+k+1}$. Let $s_2$ and $s_3$ be the cut-vertices of $H_u$ and $H_v$, respectively (if any of these cut-vertices didn't exist, for instance $s_1$, then the corresponding subgraph, $H_t$, is disconnected and the procedure is the same as  subcase 2.1) and $|V(H_u)|=3n/2+k^{\prime}$ and $|V(H_v)|=3n/2+k^{\prime\prime}$, where $k^{\prime}, k^{\prime\prime}\in\{0,1,2, \ldots, n-2\}$. Also, let $X_u$, $Y_u$ and $X_v$, $Y_v$ be the set of vertices of two components of $H_u-s_2$ and $H_v-s_3$, respectively. Since, $H_t-s_1$, $H_u-s_2$ and $H_v-s_3$ are disconnected, with the same statement of subcase 2.1 and without loss of generality, we have $Y_t\cap X_u=Y_t\cap X_v=X_t\cap X_u=X_t\cap X_v=X_u\cap X_v=\emptyset$. Hence $|V(G)|\geq{|V(H_t-s_1)|+|X_u|+|X_v|}$ which means $|V(G)|\geq{(3n/2+k-1)+(n/2+k^{\prime}+1)+(n/2+k^{\prime\prime}+1)}>5n/2-1$,
 which is a contradiction, and this completes the proof.\\
 
 Now, by Theorem \ref{tn} and \ref{t1}, the following Corollary is obvious.
 \begin{corollary}
 For $n\geq{6}$ and even, we have
 $R(W_n,S_n)=5n/2-2$ or $5n/2-1$.
 \end{corollary}
 
 
 




\bibliographystyle{plain} 
\bibliography{mybib}{}

\begin{thebibliography}{11}
\bibitem{K} A. Korolova, Ramsey numbers of stars versus wheels of similar sizes, Discrete Math. 292 (2005) 107-117
\bibitem{H} V. Chv\'{a}tal, F. Harary, Generalized Ramsey theory for graphs, III. Small off-diagonal numbers, Pacific J. Math. 41 (1972) 335-345.
\bibitem{S1} Surahmat, E.T. Baskoro, H.J. Broersma, The Ramsey numbers of large star-like trees versus large odd wheels,
University of Twende Memorandum No. 1621, March 2002.
\bibitem{S2} Surahmat, E.T. Baskoro, On the Ramsey number of a path or a star versus $W_4$ or $W_5$, in: Proceedings of
the 12th Australasian Workshop on Combinatorial Algorithms, Bandung, Indonesia, July 1417, 2001, pp.
174-179.
\bibitem{C}Y. Chen, Y. Zhang, K. Zhang, The Ramsey numbers of stars versus wheels. Eur. J. Comb. 25  (2004) 1067-
1075.
\bibitem{HB} E.T. Hasmawati, E.T. Baskoro, H. Assiyatun, Star–wheel Ramsey numbers. J. Comb. Math. Comb.
Comput. 55 (2005) 123-128.
\bibitem{Ha} J.M. Hasmawati, Bilangan Ramsey untuk graf bintang terhadap graf roda. Tesis Magister, Departemen
Matematika ITB, Indonesia (2004).
\bibitem{Y} Y. Zhang, Y. Chen, K. Zhang, The Ramsey numbers for stars of even order versus a wheel of order nine, European J. combin. 29 (2008) 1744-1754.
\bibitem{Z} Y. Zhang, T.C.E. Cheng, Y. Chen, The Ramsey numbers for stars of odd order versus a wheel of order
nine. Discrete Math. Algorithm Appl. 1(3), 123  (2009) 413-436.
\bibitem{BS} B. Li, I. Schiermeyer, On Star–Wheel Ramsey Numbers. Graphs and Combinatorics. (2015) 1-7
\bibitem{B} S. Brandt, R. Faudree, W. Goddard, Weakly pancyclic graphs, J. Graph Theory.  27 (1998) 141-176.
\bibitem{D} G.A. Dirac, Some theorems on abstract graphs, Proc. Lond. Math. Soc. 2 (1952) 69-81.
\bibitem{F} R.J. Faudree, R.H. Schelp, All Ramsey numbers for cycles in graphs, Discrete Math. 8 (1974) 313-329.
\bibitem{R} V. Rosta, On a Ramsey type problem of J.A. Bondy and P. Erd\"{o}s, I \& II, J. Combin. Theory Ser. B 15 (1973) 94-120.
\bibitem{GO}G. Chartrand, O. R. Oellermann, Applied and Algorithmic Graph Theory, Mc Graw-Hill, Inc, 1993.
\end{thebibliography}

\end{document}